\pdfoutput=1
\documentclass[a4paper,12pt]{amsart}
\usepackage{amssymb,amsfonts,latexsym}
\title{Asymptotic Expansions for the multiple gamma functions of Barnes-Milnor type}
\author{Hanamichi Kawamura}
\address{Seifu High School, 12-16, Ishigatsuji\-cho, Tennoji\-ku, Osaka\-shi, Osaka\-fu 543-0031, Japan}
\email{multiplegamma@outlook.com}
\begin{document}
\newtheorem{Thm}{Theorem}[section]
\newtheorem{Def}[Thm]{Definition}
\newtheorem{Lem}[Thm]{Lemma}
\newtheorem{Prop}[Thm]{Proposition}
\begin{abstract}
The classical Stirling's formula gives the asymptotic behavior of the gamma function. Katayama and Ohtsuki generalized this formula for Barnes' multiple gamma functions. In this paper, we further generalize these formulas for the multiple gamma functions of BM (Barnes-Milnor) type. 
\end{abstract}
\maketitle
\section{Introduction}
The multiple gamma functions were introduced by Barnes \cite{barnes}. His idea is that the multiple Hurwitz zeta functions can be applied to Lerch's formula for defining the multiple gamma functions. After his discovery, many mathematicians have studied this function. Among them, Kurokawa-Ochiai \cite{ki} is remarkable in that they constructed a theory of the generalized gamma functions of BM (Barnes-Milnor) type. While, Katayama-Ohtsuki \cite{katayama1998} proved asymptotic expansions of the Barnes multiple gamma functions which are generalizations of Stirling's formula. Our purpose in this paper is discovering generalizations of Stirling's approximation for the BM multiple gamma functions.

Let $w,\omega_1,\omega_2,\cdots,\omega_r$ be complex numbers with positive real parts. We recall the definition of the multiple Hurwitz zeta functions $\zeta_r$ by
\begin{equation*}\displaystyle\zeta_r(s,w;{\boldsymbol{\omega}}):=\sum_{{\bf n}\geq{\bf 0}} ({\bf n}\cdot{\boldsymbol{\omega}}+w)^{-s}\end{equation*}
where ${\bf 0}=(0,\cdots,0)$, ${\bf n}=(n_1,\cdots,n_r)$, ${\boldsymbol{\omega}}=(\omega_1,\cdots,\omega_r)$, $\boldsymbol{n}\geq\boldsymbol{0}\overset{\mathrm{def}}{\Leftrightarrow}n_i\geq{0}\,(i=1,\cdots,r)$ and ${\bf n}\cdot{\boldsymbol{\omega}}=n_1\omega_1+\cdots+n_r\omega_r$. This series converges absolutely and uniformly if $\mathrm{Re}(s)>r$. The multiple Hurwitz zeta functions $\zeta_r$ are continued analytically to meromorphic functions in the whole complex plane and holomorphic except for simple poles at $s=1,\cdots,r$.\\

Put
\begin{eqnarray*}\displaystyle\Gamma_r(w;{\boldsymbol{\omega}}):=\exp\left(\left.\frac{\partial}{\partial s}\zeta_r(s,w;{\boldsymbol{\omega}})\right|_{s=0}\right).\end{eqnarray*}

This functions $\Gamma_r$ are meromorphic functions with simple poles at $w=-\boldsymbol{n}\cdot\boldsymbol{\omega}\,(\boldsymbol{n}\geq{\boldsymbol{0}})$ and have no zeros. The usual gamma function $\Gamma(w)$ can be written as $\Gamma(w)=\sqrt{2\pi}\Gamma_1(w,1)$.\\

As we mentioned before, we give generalizations of asymptotic expansions of the Barnes multiple gamma functions $\Gamma_r$ in this paper. The same method is also applicable to the generalized gamma functions defined by
\begin{eqnarray*}\displaystyle \Gamma_{r,k}(w;{\boldsymbol{\omega}})=\exp\left(\left.\frac{\partial}{\partial s}\zeta_r(s,w;{\boldsymbol{\omega}})\right|_{s=-k}\right)\qquad (k\in{\mathbb{Z}_{\geq{0}}}).\end{eqnarray*}
This function can generalize Kinkelin's formula and that was found in \cite{ki}. Moreover, we define the modified BM gamma functions in order to write our main theorem more simply:
\begin{eqnarray*}\displaystyle P_r(k,w;{\boldsymbol{\omega}})=\frac{1}{2\pi i}\int_{I(\lambda,\infty)} f_{\boldsymbol{\omega}}(t)e^{-wt}t^{-k-1}\log t\,dt+(\gamma-\pi i)a_{r,k}(w;{\boldsymbol{\omega}})\end{eqnarray*}
for $k\geq{-r}$. We check the definitions of some symbols in the next section.
\begin{Thm}
We have asymptotically for large $|w|$
\begin{equation*}\displaystyle P_{r+l}(k,w+a;({\boldsymbol{\omega}},{\boldsymbol{\alpha}}))=\sum_{N=-l}^{r+k} a_{l,N}(a;{\boldsymbol{\alpha}})P_{r}(k-N,w;{\boldsymbol{\omega}})+O\left(\frac{1}{w}\right).\end{equation*}
where 
\begin{eqnarray*}\displaystyle\frac{e^{-wt}}{\prod_{j=1}^r (1-e^{-\omega_jt})}=\sum_{n\geq{-r}} a_{r,n}(w;{\boldsymbol{\omega}})t^n\end{eqnarray*}
defines the multiple Bernoulli polynomials $a_{r,n}$.
\end{Thm}
We give a proof of the main theorem in (3) of Theorem \ref{main}. Some relations of the multiple Bernoulli polynomials $a_{r,n}(w,{\boldsymbol{\omega}})$ given by convolution plays an essential role in the above theorem. We give some properties of $a_{r,n}(w,{\boldsymbol{\omega}})$ including these relations in the next section.\\

In the case of $k=0$ in Theorem 1.1, we can get generalized Stirling's formulas: 
\begin{eqnarray*}\displaystyle\log\Gamma_{r+l}(w+a;({\boldsymbol{\omega}},{\boldsymbol{\alpha}}))=\sum_{N=-l}^{r} a_{l,N}(a;{\boldsymbol{\alpha}})P_{r}(-N,w;{\boldsymbol{\omega}})+O\left(\frac{1}{w}\right).\end{eqnarray*}

In particular, the case of $r=0, l=1, \omega_1=1$ is
\begin{eqnarray*}\displaystyle \log\Gamma_1(w+a;1)=\sum_{N=-1}^0 a_{1,N}(a;1)P_0(-N,w;\emptyset)+O\left(\frac{1}{w}\right).\end{eqnarray*}

Here, by using a well-known fact about the Laurent expansion of $\Gamma(s)$ it follows:
\begin{eqnarray*}\displaystyle a_{1,N}(w;1)&=&(-1)^{N+1}\frac{B_{N+1}(w)}{(N+1)!},\\P_0(n,w;\emptyset)&=&\frac{(-1)^n}{n!}(H_n-\log w)w^n,\end{eqnarray*}
where $B_n(w)$ are the usual Bernoulli polynomials and $H_n=\sum_{j=1}^n j^{-1}$ is the $n$-th harmonic number.
Hence we get the classical Stirling's formula
\begin{eqnarray*}\displaystyle\log\Gamma_1(w+a;1)&=&\log\Gamma(w+a)-\log\sqrt{2\pi}\\&=&\left(w+a-\frac{1}{2}\right)\log w-w+O\left(\frac{1}{w}\right).\end{eqnarray*}
\section{Proof of Theorem 1.1}
Our plan to prove the main theorem is smilar to the way to consider analytic continuations and the special values at negative integers of $\zeta_r(s,w;{\boldsymbol{\omega}})$. Therefore, we need the integral representation of the BM multiple gamma functions like that of the Barnes multiple gamma functions
\begin{eqnarray*}\displaystyle\log\Gamma_r(w,{\boldsymbol{\omega}})=\frac{1}{2\pi i}\int_{I(\lambda,\infty)} f_{\boldsymbol{\omega}}(t)e^{-wt}\frac{\log t}{t}\,dt+(\gamma-\pi i)\zeta_r(0,w,{\boldsymbol{\omega}}),\end{eqnarray*}
where $f_{\boldsymbol{\omega}}(t)=\prod_{j=1}^r (1-e^{-\omega_jt})^{-1}$, $0<\lambda<\min_{1\leq i\leq r}(|2\pi/\omega_{i}|)$ and $I(\lambda, \infty)$ is the path consisting of the infinite line from $\infty$ to $\lambda$, the circle of radius $\lambda$ around $0$ in the positive sense and the infinite line from $\lambda$ to $\infty$. This integral converges if $\mathrm{Re}(w)>0$.
\begin{Prop}
The following are true:\\
\begin{description}
\item[(1)]\mbox{}
\begin{align*}\displaystyle a_{r,n}(w;{\boldsymbol{\omega}})=\frac{1}{|{\boldsymbol{\omega}}|_{\times}}\sum_{n_0,\cdots,n_r\geq{0}\atop{n_0+\cdots+n_r=n+r}} \frac{w^{n_0}B_{n_1}(w)B_{n_2}(w)\cdot\cdots\cdot B_{n_r}(w)}{n_0!\cdot\cdots\cdot n_r!}\\\quad\cdot(|{\boldsymbol{\omega}}|-1)^{n_0}(-\omega_1)^{n_1}(-\omega_2)^{n_2}\cdot\cdots\cdot(-\omega_r)^{n_r}.\end{align*}
\item[(2)]\mbox{}
\begin{eqnarray*}\displaystyle\zeta_r(-n,w,{\boldsymbol{\omega}})=(-1)^nn!a_{r,n}(w;{\boldsymbol{\omega}}).\end{eqnarray*}
\item[(3)]\mbox{}
\begin{eqnarray*}\displaystyle\frac{\partial}{\partial w}a_{r,-n}(w;{\boldsymbol{\omega}})=-a_{r,n-1}(w;{\boldsymbol{\omega}}).\end{eqnarray*}
\item[(4)]\mbox{}
\begin{eqnarray*}\displaystyle a_{r+l,k}(w;({\boldsymbol{\omega}},{\boldsymbol{\alpha}}))=\sum_{N=-l}^{r+k} a_{r,k-N}(a;{\boldsymbol{\omega}})a_{l,N}(b;{\boldsymbol{\alpha}}).\end{eqnarray*} 
\end{description}
where $|{\boldsymbol{\omega}}|=\sum_{i=1}^r \omega_i$, $|{\boldsymbol{\omega}}|_{\times}=\prod_{i=1}^r \omega_i$, $a,b$ are arbitrary complex numbers that satisfy $a+b=w$, and $B_n(w)=(-1)^nn!a_{1,n-1}(w,1)$ is the usual Bernoulli polynomial.
\end{Prop}
\begin{proof}
\begin{description}
\item[(1)]\mbox{}
From a well-known representation
\begin{eqnarray*}\displaystyle\frac{te^{wt}}{e^t-1}=\sum_{n=0}^{\infty} \frac{B_n(w)}{n!}t^n,\end{eqnarray*}
it follows that
\begin{eqnarray*}\displaystyle &&f_{\boldsymbol{\omega}}(t)e^{-wt}\\&=&\frac{1}{1-e^{-\omega_1t}}\cdot\cdots\cdot\frac{1}{1-e^{-\omega_rt}}e^{-wt}\\&=&\frac{1}{|{\boldsymbol{\omega}}|_{\times}}\frac{-\omega_1te^{-w\omega_1t}}{e^{-\omega_1t}-1}\cdot\cdots\cdot\frac{-\omega_rte^{-w\omega_rt}}{e^{-\omega_rt}-1}t^{-r}e^{(|{\boldsymbol{\omega}}|-1)wt}\\&=&\frac{1}{|{\boldsymbol{\omega}}|_{\times}}\sum_{n_0,\cdots,n_r\geq{0}} \frac{w^{n_0}B_{n_1}(w)B_{n_2}(w)\cdot\cdots\cdot B_{n_r}(w)}{n_0!\cdot\cdots\cdot n_r!}\\&{}&\quad\cdot(|{\boldsymbol{\omega}}|-1)^{n_0}(-\omega_1)^{n_1}(-\omega_2)^{n_2}\cdot\cdots\cdot(-\omega_r)^{n_r}t^{n_0+\cdots+n_r}.\end{eqnarray*}
\item[(2)]\mbox{}
We partition the integral representation
\begin{eqnarray*}\displaystyle\zeta_r(s,w,{\boldsymbol{\omega}})=\frac{1}{\Gamma(s)}\int_0^{\infty} f_{\boldsymbol{\omega}}(t)e^{-wt}t^{s-1}\,dt\end{eqnarray*}
as follows:
\begin{eqnarray*}\displaystyle\Gamma(s)\zeta_r(s,w,{\boldsymbol{\omega}})&=&I_1(s)+I^n_2(s)+I^n_3(s)\\I_1(s)&=&\int_1^{\infty} f_{\boldsymbol{\omega}}(t)e^{-wt}t^{s-1}\,dt\\I^n_2(s)&=&\int_0^1 \left(\sum_{k=-r}^{n} a_{r,k}(w;{\boldsymbol{\omega}})t^k\right)t^{s-1}\,dt\\I^n_3(s)&=&\int_0^1 \left(f_{\boldsymbol{\omega}}(t)e^{-wt}-\sum_{k=-r}^{n} a_{r,k}(w;{\boldsymbol{\omega}})t^k\right)t^{s-1}\,dt.\end{eqnarray*}
Then $I_1$ is an entire function and $I_3^n$ is holomorphic if $\mathrm{Re}(s)>-n-1$, hence we have
\begin{eqnarray*}\displaystyle\zeta_r(-n,w;{\boldsymbol{\omega}})&=&\lim_{s\rightarrow{-n}}\frac{1}{\Gamma(s)}(I_1(s)+I^n_2(s)+I^n_3(s))\\&=&\lim_{s\rightarrow{-n}} \frac{1}{\Gamma(s)}\int_0^1 \left(\sum_{k=-r}^{n} a_{r,k}(w;{\boldsymbol{\omega}})t^k\right)t^{s-1}\,dt\\&=&\lim_{s\rightarrow{-n}}\frac{1}{\Gamma(s)}\sum_{k=-r}^n \frac{a_{r,k}(w;{\boldsymbol{\omega}})}{s+k}\\&=&(-1)^nn!a_{r,n}(w;{\boldsymbol{\omega}}).\end{eqnarray*}
\item[(3)]\mbox{}
The statement can be obtained immediately from the following: 
\begin{eqnarray*}\displaystyle\sum_{n=-r}^{\infty} \frac{\partial}{\partial w}a_{r,n}(w;{\boldsymbol{\omega}})t^n&=&\frac{\partial}{\partial w}e^{-wt}\prod_{i=1}^r (1-e^{-\omega_it})^{-1}\\&=&-te^{-wt}\prod_{i=1}^r (1-e^{-\omega_it})^{-1}\\&=&-\sum_{n=-r}^{\infty} a_{r,n}(w;{\boldsymbol{\omega}})t^{n+1}\\&=&-\sum_{n=1-r}^{\infty} a_{r,n-1}(w;{\boldsymbol{\omega}})t^n\end{eqnarray*}
\item[(4)]\mbox{}
We can get the statement from
\begin{eqnarray*}\displaystyle e^{-at}\prod_{i=1}^r (1-e^{-\omega_it})^{-1}&=&\sum_{L\geq{-r}} a_{r,L}(a;{\boldsymbol{\omega}})t^L\\e^{-bt}\prod_{j=1}^l(1-e^{-\alpha_jt})^{-1}&=&\sum_{M\geq{-l}} a_{l,M}(b;{\boldsymbol{\alpha}})t^M.\end{eqnarray*}
This follows from
\begin{eqnarray*}\displaystyle \sum_{N\geq{-r-l}} a_{r+l,N}(w;({\boldsymbol{\omega}},{\boldsymbol{\alpha}}))=\sum_{L\geq{-r},M\geq{-l}}a_{r,L}(a;{\boldsymbol{\omega}})a_{l,M}(b;{\boldsymbol{\alpha}})t^{L+M}\end{eqnarray*}
\end{description}
\end{proof}
\begin{Thm}\label{main}
The following are true:
\begin{description}
\item[(1)]\mbox{}
\begin{eqnarray*}\displaystyle P_r(k,w;{\boldsymbol{\omega}})&=&\frac{(-1)^kk!}{2\pi i}\int_{I(\lambda,\infty)} f_{\boldsymbol{\omega}}(t)e^{-wt}t^{-k-1}\log t\,dt\\&{}&+(\gamma-\pi i-H_k)a_{r,k}(w;{\boldsymbol{\omega}}).\end{eqnarray*}
\item[(2)]\mbox{}
\begin{eqnarray*}\displaystyle P_r(k,w;{\boldsymbol{\omega}})-P_r(k,w+\omega_i;{\boldsymbol{\omega}})=P_{r-1}(k,w;{\boldsymbol{\omega}}\langle{i}\rangle).\end{eqnarray*}
\item[(3)]\mbox{}
\begin{eqnarray*}\displaystyle P_{r+l}(k,w+a;({\boldsymbol{\omega}},{\boldsymbol{\alpha}}))=\sum_{N=-l}^{r+k} a_{l,N}(a;{\boldsymbol{\alpha}})P_{r}(k-N,w;{\boldsymbol{\omega}})+O\left(\frac{1}{w}\right).\end{eqnarray*}
\end{description}
where $H_k$ is the $k$-th harmonic number $\sum_{j=1}^k j^{-1}$, $\gamma$ is the Euler constant,\\and ${\boldsymbol{\omega}}\langle{i}\rangle=(\omega_1,\cdots,\omega_{i-1},\omega_{i+1},\cdots,\omega_r)$.
\end{Thm}
\begin{proof}
\begin{description}
\item[(1)]\mbox{}
We can derive the statement from
\begin{eqnarray*}\displaystyle \lim_{s\rightarrow{-k}} \left(\frac{\Gamma'}{\Gamma}(s)+\frac{2\pi ie^{2\pi is}}{e^{2\pi is}-1}\right)=H_k+\pi i-\gamma.\end{eqnarray*}
\item[(2)]\mbox{}
\begin{eqnarray*}\displaystyle&&P_r(k,w;{\boldsymbol{\omega}})-P_r(k,w+\omega_i;{\boldsymbol{\omega}})\\&=&\frac{1}{2\pi i}\int_{I(\lambda,\infty)} f_{\boldsymbol{\omega}}(t)e^{-wt}(1-e^{-\omega_it})t^{-k-1}\log t\,dt\\&{}&\quad+(\gamma-\pi i)a_{r,k}(w;{\boldsymbol{\omega}})-(\gamma-\pi i)a_{r,k}(w+\omega_i;{\boldsymbol{\omega}})\\&=&P_{r-1}(k,w;{\boldsymbol{\omega}})\\&{}&\quad+(\gamma-\pi i)(a_{r,k}(w;{\boldsymbol{\omega}})-a_{r,k}(w+\omega_i;{\boldsymbol{\omega}})-a_{r-1,k}(w;{\boldsymbol{\omega}})).\end{eqnarray*}
Thus we only have to show
\begin{eqnarray*}\displaystyle a_{r,k}(w;{\boldsymbol{\omega}})-a_{r,k}(w+\omega_i;{\boldsymbol{\omega}})=a_{r-1,k}(w;{\boldsymbol{\omega}}),\end{eqnarray*}
which follows from
\begin{eqnarray*}\displaystyle&&\sum_{N\geq{-r}}(a_{r,N}(w;{\boldsymbol{\omega}})-a_{r,N}(w+\omega_i;{\boldsymbol{\omega}}))t^N\\&=&f_{\boldsymbol{\omega}\langle{i}\rangle}(t)e^{-wt}\\&=&\sum_{N\geq{1-r}} a_{r-1,N}(w;{\boldsymbol{\omega}}\langle{i}\rangle).\end{eqnarray*}
\item[(3)]\mbox{}
We can get easily
\begin{eqnarray*}\displaystyle&&P_{r+l}(k,w+a;({\boldsymbol{\omega}},{\boldsymbol{\alpha}}))=(\gamma-\pi i)a_{r+l,k}(w+a;({\boldsymbol{\omega}},{\boldsymbol{\alpha}}))\\&{}&+\sum_{N=-l}^{r+k} a_{l,N}(a;{\boldsymbol{\alpha}})(P_l(k-N,w;{\boldsymbol{\omega}})-(\gamma-\pi i)a_{l,k-N}(w;{\boldsymbol{\omega}}))\\&{}&+\frac{(-1)^kk!}{2\pi i}\int_{I(\lambda,\infty)} f_{\boldsymbol{\omega}}(t)e^{-wt}\biggl(f_{\boldsymbol{\alpha}}(t)e^{-at}\biggr.\\&{}&\biggl.-\sum_{N=-l}^{r+k} a_{l,N}(a;{\boldsymbol{\alpha}})t^N\biggr)t^{-k-1}\log t\,dt.\end{eqnarray*}
Since
\begin{equation*}\displaystyle f_{\boldsymbol{\alpha}}(t)e^{-at}-\sum_{N=-l}^{r+k} a_{l,N}(a;{\boldsymbol{\alpha}})t^N=O(t^{r+k+1})\qquad{(t\rightarrow{0})},\end{equation*}
the third term is an entire function and has the root at $t=0$. Hence, it follows by using (4) of Proposition 2.1:
\begin{eqnarray*}\displaystyle&&P_{r+l}(k,w+a;({\boldsymbol{\omega}},{\boldsymbol{\alpha}}))\\&=&\sum_{N=-l}^{r+k} a_{l,N}(a;{\boldsymbol{\alpha}})P_{r}(k-N,w;{\boldsymbol{\omega}})\\&+&\int_0^{\infty} f_{\boldsymbol{\omega}}(t)e^{-wt}\left(f_{\boldsymbol{\alpha}}(t)e^{-at}-\sum_{N=-l}^{r+k} a_{l,N}(a;{\boldsymbol{\alpha}})t^N\right)t^{-k-1}\,dt.\end{eqnarray*}
The proposition from the series expansion of the second term.
\end{description}
\end{proof}

\end{document}